\documentclass[12pt,reqno]{amsart}
\usepackage{amssymb}
\usepackage{ascmac}
\usepackage{amsmath}
\usepackage{amsthm}
\usepackage{cases}
\usepackage{color}
\usepackage[dvipdfmx]{xcolor}
 \usepackage[a4paper, margin=55pt]{geometry}
\usepackage{hyperref}
\usepackage{shuffle}
\usepackage{enumerate}
 \newtheoremstyle{mystyle}
    {}
    {}
    {\normalfont}
    {}
    {\bfseries}
    {}
    { }
    {}
\theoremstyle{mystyle}
\newtheorem{ex}{Example}[section]

\newtheorem{df}[ex]{Definition}
\newtheorem{lem}[ex]{Lemma}
\newtheorem{prop}[ex]{Proposition}
\newtheorem{thm}[ex]{Theorem}
\newtheorem{rmk}[ex]{Remark}

\newtheorem{coro}[ex]{Corollary}

\newcommand{\qsp}{\ast_\diamond}
\theoremstyle{definition}
\newcommand{\Z}{\mathbb{Z}} 
\newcommand{\Q}{\mathbb{Q}}
\newcommand{\R}{\mathbb{R}}

\newcommand{\mz}{\mathcal{Z}}

\newcommand{\bg}{\begin} 
\newcommand{\en}{\end} 
\newcommand{\bsb}{\begin{shadebox}}
\newcommand{\esb}{\end{shadebox}}
\newcommand{\ijo}{\geq}
\newcommand{\ika}{\leq} 
\newcommand{\seki}{\prod}
\newcommand{\tenten}{\cdots}
\newcommand{\bal}{\begin{align}}
\newcommand{\eal}{\end{align}}
\newcommand{\bali}{\begin{align*}}
\newcommand{\eali}{\end{align*}}

\newcommand{\bgyo}{\begin{pmatrix}}
\newcommand{\egyo}{\end{pmatrix}}
\newcommand{\bshi}{\begin{vmatrix}}
\newcommand{\eshi}{\end{vmatrix}}
\newcommand{\Li}{\operatorname{Li}}
\newcommand{\qgerH}{\widehat{\gerH}}

\newcommand{\gerH}{\mathfrak{H}}
\newcommand{\benu}{\begin{enumerate}[i)]}
\newcommand{\eenu}{\end{enumerate}}

\newcommand{\ds}{\operatorname{ds}}
\newcommand{\dia}{\diamond}
\newcommand{\CLL}{\mathcal{C}\langle L\rangle}
\newcommand{\CL}{\mathcal{C}L}
\newcommand{\CC}{\mathcal{C}}
\newcommand{\ew}{{\bf 1}}
\newcommand{\gsh}{\hat{\shuffle}}
\newcommand{\RLL}{\mathcal{R} \langle L \rangle}
\newcommand{\RL}{\mathcal{R}  L }
\numberwithin{equation}{section}
\newcommand{\rqsp}{\overline{\ast_{\diamond}}}
\newcommand{\rgsh}{\overline{\shuffle}}
\begin{document}
 




\title{Generalized quasi-shuffle products}
\author{Masataka Satoh}
\address{Graduate School of Mathematics, Nagoya University, Nagoya, Japan.}
\email{masatakasato185@gmail.com}



\begin{abstract}
In this paper, we introduce the notion of generalized quasi-shuffle products and give a criterion for their associativity. These extend the quasi-shuffle products introduced by Hoffman, which are often used to describe the stuffle and shuffle product for multiple zeta values. For $q$-analogues of multiple zeta values, the description of an analogue for the shuffle product can often not be described with the classical notion of quasi-shuffle products. We show that our generalization gives a natural extension to also include these types of products and we prove a generalization of a duality between the $q$-shuffle product and the $q$-stuffle product. 
\end{abstract}

\maketitle

\section{Introduction}


In this work, we will discuss the algebraic structure of multiple zeta values and their $q$-analogues. The main result is the introduction of generalized quasi-shuffle products, which we will use to describe the shuffle product for $q$-analogues and to give a new interpretation of the so-called dual stuffle product of multiple zeta values. The results obtained in this work are based on the authors master thesis at Nagoya University. 

\emph{Multiple zeta values} are real numbers defined for $k_1\geq 2, k_2,\dots,k_r\geq 1$ by 
\begin{align*}
\zeta (k_1,\dots,k_r) = \sum_{m_1 > m_2 > \cdots > m_r > 0} \frac{1}{m_1 ^{k_1}\cdots m_r ^{k_r} } \in \R\,.
\end{align*}
In particular, for $r=1$ the $\zeta(k_1)$ are the \emph{Riemann zeta values}. 
The multiple zeta values in depth two were first considered by Goldbach and Euler (\cite{K}). In 1776s, the multiple zeta values were introduced by Euler (\cite{E}). In 1992s, Hoffman defined multiple zeta values for any depth as above (\cite{H3}).
The multiple zeta values satisfy various relations \big(e.g., the \emph{double shuffle relation}. (\cite{AK}, \cite{IKZ}), the \emph{duality relation}. (\cite[p. 510]{Z2}, \cite[p. 13]{AK})\big). In addition, among them, there are also known relations obtained through knot invariants (e.g., \cite{LM}).
Furthermore, it is known that the multiple zeta values also appear in mathematical physics (e.g., \cite{BK}). In the following, we introduce some of the properties of multiple zeta values.

The product of two multiple zeta values can be expressed as a linear combination of multiple zeta values. For example, by the definition as an iterated sum, we have
\begin{align}\label{eq:stufflez2z3}
\begin{split}
    \zeta(2) \zeta(3) &= \sum_{m = 1}^{\infty} \frac{1}{{m}^{2}} \sum_{n = 1}^{\infty} \frac{1}{{n}^{3}}
    =\left( \sum_{m > n > 0} + \sum_{n > m > 0} + \sum_{m=n>0} \right) \frac{1}{{m}^{2} {n}^{3}} \\
    &= \zeta(2,3) + \zeta(3,2) + \zeta(5)\,.
    \end{split}
\end{align}
To describe this product algebraically, one can use the notion of quasi-shuffle products introduced by Hoffman (\cite{H},\cite{HI}). We will recall this classical notion of quasi-shuffle products in Section \ref{sec:qsh}, where we introduce a $\Q$-vector space $\gerH^{0}$ spanned by words in the letters $z_k$ for $k\geq 1$, which do not start with $z_1$. For example, an element in $\gerH^{0}$ is the word $z_3 z_4 z_2$, which will correspond to the multiple zeta value $\zeta(3,4,2)$. The space $\gerH^{0}$ can be equipped with the \emph{stuffle product} $\ast$ and we have, for example,
\begin{align}\label{al:ast}
z_2 \ast z_3&=z_2z_3+z_3z_2+z_5\,.
\end{align}
But since multiple zeta values can also be expressed as iterated integrals, we also have 
\begin{align}\label{eq:shufflez2z3}
       \zeta(2) \zeta(3) =3\zeta(3,2)+6\zeta(4,1)+\zeta(2,3).
\end{align}
This can also be described by using the theory of quasi-shuffle products, by defining the \emph{shuffle product} $\shuffle$ on $\gerH^0$. For example, 
\begin{align}\label{al:sh}
z_2\shuffle z_3 &= 3z_3 z_2 + 6z_4 z_1+z_2 z_3.
\end{align}
We then can view the multiple zeta values as a $\Q$-linear map $\zeta : \gerH^{0} \rightarrow \R$, which is a algebra homomorphism from $\gerH^{0}$ to $\mz$ with respect to both products $\ast$ and $\shuffle$. In particular, applying $\zeta$ to \eqref{al:ast} we obtain  \eqref{eq:stufflez2z3} and from \eqref{al:sh} we deduce \eqref{eq:shufflez2z3}.
we describe double shuffle relation
by comparing $\ast$ with $\shuffle$ and
define the \emph{double shuffle} element $\ds(w,v) = w \shuffle v - w \ast v \in \ker \zeta$ for $w,v \in \gerH^0$. Conjecturally all relations among multiple zeta values are a consequence of the double shuffle relations after some regularization process (\cite[p. 51]{AK}, \cite[Conjecture 1]{IKZ}).
We also define an involution $\tau: \gerH^0 \rightarrow \gerH^0$ and show that $\zeta \circ \tau = \zeta$, which is called the \emph{duality relation}. 
In Section \ref{chapter:qmzv}, we discuss the algebraic structure of $q$-analogues of multiple zeta values. In general, by a $q$-analogue, we mean some $q$-series, which evaluates to multiple zeta values as $q\rightarrow 1$.
In recent years, different models of $q$-analogues of multiple zeta values were studied (e.g., \cite{Ba0}, \cite{CEM}, \cite{EMS}, \cite{T0} \cite{Zh}). In this work, we consider the Schlesinger-Zudilin $q$-multiple zeta values defined for $k_1\geq 1, k_2,\dots,k_r \geq 0$ by 
\begin{align*}
\zeta_{q}(k_{1},\dots, k_{r}) = \sum_{m_{1} > \tenten > m_{r} > 0} \prod_{j=1}^r \frac{q^{k_j m_j }}{[m_j]_{q}^{k_{j}}},
 \end{align*} 
 where $[m]_q=\frac{1-q^m}{1-q}$. Here we see that for $k_1\geq 2, k_2,\dots,k_r\geq 1$ we have $\lim\limits_{q\rightarrow 1} \zeta_{q}(k_{1},\dots, k_{r})=\zeta(k_1,\dots,k_r)$. 
 Notice that by definition, the product of two $q$-analogues can also be expressed by the stuffle product formula, e.g., similar to \eqref{eq:stufflez2z3}
\begin{align*}
    \zeta_q(2) \zeta_q(3) = \zeta_q(2,3) + \zeta_q(3,2) + \zeta_q(5)\,.
\end{align*}

To describe this product algebraically, we will define the \emph{$q$-stuffle product} $*_{q}$ as a quasi-shuffle product on a space $\qgerH^0$, which can be seen as an extension of $\gerH^0$.

Since for multiple zeta values, the double shuffle relations conjecturally give all relations, it is natural to ask how this situation is for $q$-analogues and if there is an analogue of the shuffle product. To describe an analogue of the shuffle product, there have been several different approaches for different models of $q$-analogues. Some of these use Rota-Baxter operators, and Jackson integrals (e.g., \cite{CEM}, \cite{Z}), and others use more algebraic approaches (e.g., \cite{Ba0}, \cite{EMS}, \cite{T0}, \cite{T}). In analogy to the conjecture for multiple zeta values, it is expected that the corresponding $q$-shuffle relations also give all relations among $q$-analogues (c.f. \cite{Ba0}, \cite{T0}). 

In this work, we follow \cite{EMS} and \cite{T} and replace the iterated integrals with a certain action, which we will describe in Section \ref{Multiple_q-poly}. This analogue, and also all other variants for different models in the literature, can not be described by using the classical notion of quasi-shuffle products. This was one of the main motivations for this paper. 

The main result of this paper is to introduce a natural generalization of the notion of quasi-shuffle products (Definition \ref{def:generalizedqsh}) and show when this product is associative (Theorem \ref{main}).
Through the generalized quasi-shuffle product, we can describe an analogue of the shuffle product for $q$-analogues, called the $q$-shuffle product $\shuffle_q$ (Definition \ref{def:qshuffle}). For example, this will lead to the following analogue of \eqref{eq:shufflez2z3}
\begin{align*}
\zeta_q(2) \zeta_q(3) =&\,\, 3 \zeta_q(3,2)+ 6 \zeta_q(4,1) + \zeta_q(2,3) \\
&+ (1-q) \left( 2 \zeta_q(2,2) + 7 \zeta_q(3,1) + 3\zeta_q(4,0) \right)\\
&+ (1-q)^2 \left( 2 \zeta_q(3,0) + \zeta_q(2,1) \right)\,.
\end{align*}
We give some examples regarding Theorem $\ref{main}$ after defining the generalized quasi-shuffle product. They contain the integral shuffle product (c.f. \cite{T0} section 2.3).

Finally, we explain a certain duality between the $q$-stuffle product and the $q$-shuffle product, which was also discussed in \cite{EMS}. More precisely, we introduce an involution $\sigma: \qgerH \rightarrow \qgerH$ (Definition \ref{def:sigma}), such that for $w,v \in \qgerH^{0}$, we have
\begin{align*}
w \shuffle_{q} v = \sigma (\sigma (w) \ast_{q} \sigma(v)).
\end{align*}

At the end, we consider a multiple zeta value version and consider the dual stuffle product, defined by $\tau(\tau(w) \ast \tau(v))$ for $w,v \in \gerH^0$. Our final result will be to show that this dual stuffle product can also be described as a generalized quasi-shuffle product. Finally, using this generalized quasi-shuffle product, we get another expression for the double shuffle relations.

 \section{Multiple zeta values \& q-analogues of multiple zeta values}
In this section, we describe an algebra setup for the multiple zeta values and their q-analogues. 
We consider their algebraic structure by using quasi-shuffle products (following \cite{H},\cite{HI}). 
Later we describe the relationship between their algebra structures and multiple zeta values. And we also describe the relationship between these algebra structures and $q$-analogues of multiple zeta values.
\newcommand{\wt}{\operatorname{wt}}
\newcommand{\dep}{\operatorname{dep}}

\subsection{Quasi-shuffle products}\label{sec:qsh}
In this section, we recall quasi-shuffle products, which were introduced by Hoffman in \cite{H} and later studied in more detail in \cite{HI}. We will use this to describe the algebraic setting of multiple zeta values. The results in this sections are well-known and they can be found in \cite{AK} and \cite{HI}.
\\
Let $\mathcal{R}$ be a commutative ring (with unit) containing $\Q$ and let $L$ be an arbitrary set. We refer to the element in $L$ as \emph{letters}. $\mathcal{R} \langle L \rangle$ is the non-commutative polynomial in the variables in $L$ with coefficients in $\mathcal{R}$. A monic monomial $\alpha_1 \dots ,  \alpha_n \in \mathcal{R} \langle L \rangle$ with $\alpha_j \in L$ is called a \emph{word} and by $\ew$ we denote the \emph{empty word}. 
Let $\dia$ be a $\mathcal{R}$-bilinear product on $\mathcal{R} L$ which is commutative and associative.

\begin{df}\label{def:classicalqsh}
For words $w,v \in \mathcal{R} \langle L \rangle$, $a,b \in L$  we define the \emph{quasi-shuffle product} $\qsp$ 
\begin{align}\label{eq:qshp}
aw \qsp bv = a (w \qsp bv) + b(aw \qsp v) + (a \dia b)(w \qsp v)
\end{align}
and  $\ew \qsp w = w \qsp \ew = w$ for any $w \in \mathcal{R} \langle L \rangle$.
\end{df}

\begin{df}\label{def:revclassicalqsh}
For words $w,v \in \mathcal{R} \langle L \rangle$, $a,b \in L$  we define the \emph{dual product} $\rqsp$ 
\begin{align}
wa \rqsp vb = (w \rqsp vb)a + (wa \rqsp v)b + (w \rqsp v)(a \dia b)
\end{align}
and  $\ew \rqsp w = w \rqsp \ew = w$ for any $w \in \mathcal{R} \langle L \rangle$.
\end{df}
One can show that these two products actually coincide:

\begin{thm}(\cite[Theorem $9$]{Zu2})\label{dualeq}
    For words $w,v \in \mathcal{R} \langle L \rangle$, we have
    $$w \qsp v = w \rqsp v.$$
\end{thm}
We omit this proof but later we will describe the more generalized version of Theorem \ref{dualeq} in Section \ref{dual:stuffle}. The next theorem (\cite{H}, \cite{HI}) states that the quasi-shuffle product is associative and commutative when $\diamond$ is associative and commutative. 
\begin{thm}(\cite[Theorem 2.1]{H}, \cite[Theorem 2.1]{HI})\label{thm:quasishufflealgebra}
Equipped with the quasi-shuffle product $\qsp$, the space $\mathcal{R}\langle L \rangle$ becomes a commutative $\mathcal{R}$-algebra.
\end{thm}
We omit this proof because we can see more generalized version of this theorem later. It is also worth mentioning that $\mathcal{R}$ in \cite{H} and \cite{HI} is a field.
Next, we will give two examples of quasi-shuffle products, namely the shuffle product and the stuffle product.
\\

Let $\Q \langle x,y \rangle $ be the non-commutative polynomial ring in the variables $x,y$.
\begin{enumerate}[i)]
 \item We define $\gerH=\Q \langle x,y \rangle $ and its $\Q$-submodules
$$ \gerH^{0} =\Q+ x \gerH y \subset \gerH^{1} =\Q + \gerH y \subset \gerH . $$
\item For $k \ijo 1$ we set $z_{k} = x^{k-1}y$. Then we have
\begin{align*}
 \gerH^{1} &= \Q \langle z_{1} ,z_2 , \dots \rangle ,\\
  \gerH^{0} &= \Q + \langle z_{k_1} \cdots z_{k_r} | r \ijo 1, k_1 \ijo 2, k_2 ,\dots, k_r \ijo 1 \rangle.
\end{align*}
\end{enumerate}

Our first example of the quasi-shuffle product will be the so-called stuffle product. For this we consider the case $\mathcal{R}=\Q$, $L=\{z_1,z_2, \cdots \}$ and $z_i \dia z_j$ be $z_{i+j}$ with $z_i,z_j \in L$. In this case we have $\gerH^1 = \mathcal{R}\langle L\rangle$.
Then the corresponding quasi-shuffle product  will we be called \emph{stuffle product}, denoted by $\ast$, and \eqref{eq:qshp} becomes
\begin{align*}
z_i w \ast z_j v = z_i (w \ast z_j v) + z_j(z_i w \ast v) + z_{i+j}(w \ast v)
\end{align*}
and  $\ew \ast w = w \ast \ew = w$ for any $w \in \Q \langle z_1,z_2, \cdots \rangle$. By Theorem \ref{thm:quasishufflealgebra} we obtain a commutative $\Q$-algebra $\gerH^1_\ast$. Notice that the subspace $\gerH^0 \subset \gerH^1$ is closed under $\ast$ and therefore we also obtain a $\Q$-algebra  $\gerH^0_\ast$.\\
Next, we will check $\zeta$ is an algebra homomorphism from $\gerH_{\ast}^{0}$ to $\R$.
\bg{df}
\label{zetamap}
We define the $\Q$-linear map $\zeta$ by
 \begin{align*}
    \begin{array}{rccc}
\zeta \colon &\gerH^{0}&\longrightarrow& \R                    \\
        & z_{k_1}\cdots z_{k_r}                    & \longmapsto   & \zeta(k_1,\cdots,k_r).
\end{array}
\end{align*}

\en{df}

\bg{prop}(\cite[Theorem 4.2]{H2}, \cite[p. 308]{IKZ})
\label{st}
For $v,w \in \gerH^{0}$ we have
\begin{align*}
\zeta(v)\zeta(w)=\zeta(v \ast w).
\end{align*}
In particular, the space $\mz$ is a $\Q$-algebra.
\en{prop}

Next, we will introduce the shuffle product.
Let $k$ be $\Q$, $L$ be $\{ x,y \}$ and $\alpha \dia \beta$ be $0$ with $\alpha,\beta \in L$. Then we define shuffle $\shuffle$ product by
\begin{align*}
\alpha w \shuffle \beta v = \alpha (w \shuffle \beta v) + \beta(\alpha w \shuffle v) 
\end{align*}
and  $\ew \shuffle w = w \shuffle \ew = w$ for any $w \in  \gerH$.

By Theorem \ref{thm:quasishufflealgebra}, we obtain a commutative $\Q$-algebra $\gerH_\shuffle$\.. Notice that both subspaces $\gerH^0 \subset \gerH^1 \subset \gerH$ are closed under $\shuffle$ and therefore we also obtain $\Q$-subalgebras  $\gerH^0_\shuffle$ and $\gerH^1_\shuffle$.\\
It is also known that that $\zeta$ is an algebra homomorphism from $\gerH_{\shuffle}^{0}$ to $\R$.

\bg{prop}(\cite[p. 309]{IKZ})
\label{sh}
For $v,w \in \gerH^{0}$ we have
\begin{align*}
\zeta(v)\zeta(w)=\zeta(v \shuffle w).
\end{align*}
\en{prop}

\begin{df}\label{def:tau}
We define the $\Q$-anti-automomorphism $\tau$ on $\gerH^{0}$ by
 \begin{align*}
 \begin{array}{rccc}
\tau \colon &\gerH&\longrightarrow& \gerH                    \\
        & x                    & \longmapsto   & y \\
        & y                    & \longmapsto   & x \\
         & \ew                    & \longmapsto   & \ew .\\ 
\end{array}
\end{align*}
\end{df}

\begin{thm}(\cite[p. 510]{Z2})
\label{dual}
For $w \in \gerH^{0}$ we have
\begin{align*}
    \zeta(w) = \zeta(\tau (w)).
\end{align*}
\end{thm}

As a consequence of Corollary \ref{st} and \ref{sh}, we obtain the following family of linear relations among multiple zeta values.
\begin{prop}(\cite[Proposition 1.4.4]{AK})
For $w,$ $v \in \gerH^{0}$ we have
\begin{align*}
\zeta (w \shuffle v - w*v)=0.
\end{align*}
\end{prop}

 \begin{df}
 We define finite double shuffle $\ds$ by follows.
 For $w,$ $v \in \gerH^{0}$,
 \begin{align*}
 \ds(w,v):=w \shuffle v - w*v.
 \end{align*}

 \end{df}
\begin{ex}
    Let $w,v \in \gerH^{0}$ be $w=z_2$, $v=z_3$. 
    Then $\ds(w,v)= 2z_3 z_2+6z_4 z_1-z_5$. Therefore, we have
    $2\zeta(3,2)+6\zeta(4,1)-\zeta(5)=0$
\end{ex}


\subsection{q-analogue of multiple zeta values}\label{chapter:qmzv}
In this section, we discuss an algebra setup for the $q$-analogue of multiple zeta values. 

In detail, we will define $q$-analogue of multiple zeta values, and view it as a map $\zeta_q$ similar as we did for multiple zeta values. For this we define a space $\qgerH$ as an analogue of $\gerH$ and introduce the $q$-stuffle product. Later, we will check the map $\zeta_q$ is $\CC$-algebra homomorphisms with respect to the $q$-stuffle product.

We introduce the algebra setup of $q$-analogues of multiple zeta values just like multiple zeta values.
\begin{df}
Let $\hbar$ be a formal variable that corresponds to $(1-q)$.
$$\mathcal{C}:=\Q [ \hbar ,\hbar^{-1} ].$$
\begin{enumerate}[i)]
 \item We define $\widehat{\gerH}=\mathcal{C}\langle a,b \rangle $ and its $\mathcal{C}$-submodules
$$ \qgerH^{0} =\mathcal{C}+ a \qgerH b \subset \qgerH^{1} =\mathcal{C} + \qgerH b \subset \qgerH . $$
\item For $k \in \Z$ with $k \ijo 0$, $e_{k} = \overbrace{ a \cdots a }^{k}b$ then we can write
\begin{align*}
 \qgerH^{1} &= \mathcal{C} \langle e_{0},e_1 , e_2 , \dots \rangle ,\\
  \qgerH^{0} &= \mathcal{C} + \langle e_{k_1} \cdots e_{k_r} | r \ijo 1, k_1 \neq 0, k_2 ,\dots, k_r \ijo 0 \rangle.
\end{align*}
\end{enumerate}
\end{df}
We endow $\Q[[q]]$ with a $\mathcal{C}$-algebra structure by letting $\hbar$ act as multiplication by $1-q$. 
\begin{df}
We define the $\mathcal{C}$-linear map $\zeta_{q}$
 \begin{align*}
    \begin{array}{rccc}
\zeta_{q} \colon &\qgerH^{0}&\longrightarrow& \Q [[q]]                    \\
        & e_{k_1}\tenten e_{k_n}                    & \longmapsto   & \zeta_{q}(k_1,\tenten,k_n)
\end{array}
\end{align*}
and set $\zeta_{q}(\ew) =1$.
\end{df}

\begin{df}\label{def:stuffle}
Let $L=\{e_{0},e_{1},e_{2},\dots \}$ and define $\dia$ by
\begin{align*}
e_{k}\dia_{q} e_{l}=e_{k+l}
\end{align*}
for $k,l \in \Z_{\ijo 0}$.
Then we define $\ast_{q}$ is quasi-shuffle product using $\dia_{q}$.
We call this $\ast_{q}$ the \emph{$q$-stuffle product}.
We can find directly that $\qgerH^0$ and $\qgerH^1$ are both closed under $\ast_q$
\end{df}

\begin{prop}(\cite[Theorem 5.2]{EMS}, \cite[Proposition 2.1]{T})
\label{qst:hom}
For $w$, $u \in \qgerH^{0}$ we have
\begin{align*}
\zeta_{q}(w)\zeta_{q}(u)=\zeta_{q}(w\ast_q u).
\end{align*}
In particular, the map $\zeta_q$ is a $\mathcal{C}$-algebra homomorphism.
\end{prop}

\subsection{Multiple q-polylogarithm}
\label{Multiple_q-poly}
We define the $q$-multiple polylogarithm for checking $\zeta_{q}$ is $\CC$-algebra homomorphism with respect to the $q$-shuffle product. We have already described $\zeta_{q}$ is $\CC$-algebra homomorphisms with respect to the $q$-stuffle product (Proposition \ref{qst:hom}), which was done in a similar way as for multiple zeta values. However, we have to use a different way from multiple zeta values (Proposition \ref{sh}) in the case of $q$-shuffle product. For this we define a $q$-analogue of the multiple polylogarithm $\Li^{q}$ and define an action action of $\qgerH^1$ on $\Q[[q, z]]$. 
First, we define $q$-multiple polylogarithm and $\mathcal{C}$-linear map.
 \begin{df}
We define extended $q$-multiple polylogarithm $\Li^{q}_{k_{1},\dots,k_{r}}$ for  $k_{1},\cdots , k_{r} \ijo 0$ by
 \begin{align*}
 \Li^{q}_{k_{1},\dots, k_{r}}(z) = \sum_{m_{1} > \tenten > m_{r} > 0} z^{m_1}\prod_{j=1}^r \frac{q^{k_j m_j }}{[m_j]_{q}^{k_{j}}}.
 \end{align*}
 \end{df}
 \begin{df}
We define the $\mathcal{C}$- linear map $\Li^{q}$
 \begin{align*}
    \begin{array}{rccc}
\Li^{q} \colon &\qgerH^{1}&\longrightarrow& \Q [[q,z]]                    \\
        & e_{k_1}\tenten e_{k_r}                    & \longmapsto   & \Li^{q}_{k_{1},\dots, k_{r}}(z).
\end{array}
\end{align*}
and set $\Li^{q}_{1} =1$.
\end{df}
Next, we define an action for checking that $\zeta_{q}$ is a $\CC$-algebra homomorphism using $q$-shuffle product in the next section. The following action was defined in \cite{T}.
\begin{df}
We define an action of $\qgerH^{1}$ on $\Q [[q,z]]$ as follows.
For all $g \in z \Q [[q,z]]$ 
\begin{enumerate}[i)]
\item $a \cdot g = (1-q) \sum_{j=1}^{\infty} g|_{z=q^{j}z}.$
\item $b \cdot g = \frac{z}{1-z} g.$
\end{enumerate}
Notice that the last letter of all elements in $\qgerH^{1}$ is "$b$". Hence, we always use ii) the first time and we can consider an action of $\qgerH^{1}$ on $\Q [[q,z]]$.
\end{df}
\begin{lem}(\cite[p. 5]{T})
For all $n \in \Z_{> 0}$, we have
\begin{align*}
a^{n-1} \Li^{q}_{1,e_{k_2},\dots, e_{k_r}}(z) = \Li^{q}_{e_{n},e_{k_2},\dots, e_{k_r}}(z).
\end{align*}
\end{lem}
\begin{proof}
We show the statement by induction on $n$.
In the case $n=1$, immediately we can get the statement.
In the case $n=s$,
\begin{align*}
a^{s-1} \Li^{q}_{1,e_{k_2},\dots, e_{k_r}}(z) 
&= a \sum_{m_{1} > \tenten > m_{r} > 0} z^{m_{1}}\frac{q^{m_{1}(s-1)}}{[m_{1}]_{q}^{(s-1)}} \prod_{j=2}^r \frac{q^{k_j m_j }}{[m_j]_{q}^{k_j}} \\
&=  (1-q) \sum_{j=1}^{\infty} \sum_{m_{1} > \tenten > m_{r} > 0} (q^{j}z)^{m_{1}} \frac{q^{m_{1}(s-1)}}{[m_{1}]_{q}^{(s-1)}} \prod_{j=2}^r \frac{q^{k_j m_j }}{[m_j]_{q}^{k_j}} \\
&=    \sum_{m_{1} > \tenten > m_{r} > 0} z^{m_1}  \prod_{j=1}^r \frac{q^{k_j m_j }}{[m_j]_{q}^{k_j}} .
\end{align*}
\end{proof}

\begin{prop}(\cite[p. 5]{T})
\label{ex11}
For all $w \in \qgerH^{1}$ we have $\Li_{w}^{q} = w 1.$
\end{prop}

\begin{proof}
For all $w \in \qgerH^{1}$, we put $w=e_{k_1} \cdots e_{k_r}$ .
We use induction on $r$.
In the case $r=1$, for all $k \in \Z_{>0}$ we have to check $\Li_{e_{k}}^{q} =e_{k}\cdot 1.$
We use induction on $k$.
In the case $k=0$, 
\begin{align*}
\Li_{e_0}^{q}(z)=\sum_{m=1}^{\infty}z^{m} =\frac{z}{1-z}=b\cdot1 \,.
\end{align*}
In the case $k=n$,
\begin{align*}
a^{n}b \cdot 1  &= a \cdot \sum_{m=1}^{\infty} z^{m} \frac{q^{n-1}m}{[m]_{q}^{n-1}} 
= (1-q) \sum_{m=1}^{\infty} \frac{q^{m}}{1-q^{m}} z^{m} \frac{q^{n-1}m}{[m]_{q}^{n-1}} 
= \sum_{m=1}^{\infty} z^{m} \frac{q^{nm}}{[m]_{q}^{n}}.
\end{align*}
Since we can get the statement when $r=1$. 
In the case $r=t$,
\begin{align*}
w \cdot 1 &= e_{k_1}(e_{k_2} \cdots e_{k_{t}})\cdot 1 
= a^{k_{1}-1}(1-q) \sum_{j=1}^{\infty} \sum_{m_{2} > \tenten > m_{t} > 0}  \frac{(q^{j}z)^{m_2+1}}{1-q^{j}z}\prod_{j=2}^r \frac{q^{k_j m_j }}{[m_j]_{q}^{k}} \\
&= a^{k_{1}-1}(1-q) \sum_{d=1}^{\infty} \sum_{j=1}^{\infty} \sum_{m_{2} > \tenten > m_{t} > 0}  (q^{j}z)^{m_2+d} \prod_{j=2}^r \frac{q^{k_j m_j }}{[m_j]_{q}^{k}} 
=  \sum_{m_{1} > \tenten > m_{t} > 0} z^{m_1}\prod_{j=1}^r \frac{q^{k_j m_j }}{[m_j]_{q}^{k}}.
\end{align*}
\end{proof}

\section{Generalized quasi-shuffle products}
\subsection{Definition and results}
In this section, we introduce the generalized quasi-shuffle product, which is the main result of this paper. Also, we organize the conditions of the generalized quasi-shuffle product. Later, we  define the $q$-shuffle product using the generalized quasi-shuffle product. Furthermore, we check that $\zeta_q$ is $\CC$-algebra homomorphism using $q$-shuffle product.
First, we introduce some notation and definitions for explaining the generalized quasi-shuffle product.\\

In the following, we recall that $\mathcal{R}$ is a commutative ring and $L$ is a letter of set.

\bg{df}
We define the $\mathcal{R}$-linear map $F$ by
 \begin{align*}
    \begin{array}{rccc}
F \colon &\RLL       &\longrightarrow& \RL + \mathcal{R} \ew                   \\
        & \alpha_{1}\cdots \alpha_{n}                    & \longmapsto   & \alpha_{1}  \\
        & \ew                   & \longmapsto   & \ew
        \end{array}
\end{align*}
and we define $F^{k}(x)$ for $k \in \Z_{>0}$ and $x \in \RLL$ by 
$$F^{k}(x) :=  \overbrace{F(F( \cdots F}^{k}(x) \cdots )).$$
In the same way as above, we define the $\mathcal{R}$-linear map $\overline{F}$ by 
 \begin{align*}
    \begin{array}{rccc}
\overline{F} \colon &\RLL       &\longrightarrow& \RL + \mathcal{R} \ew                   \\
        & \alpha_{1}\cdots \alpha_{n}                    & \longmapsto   & \alpha_{n}  \\
        & \ew                   & \longmapsto   & \ew
        \end{array}
\end{align*}
\en{df}

\bg{df}
We define the map $R$ by
 \begin{align*}
    \begin{array}{rccc}
R \colon &\RLL       &\longrightarrow& \RLL              \\
        & \sum \limits_{i=1}^{k} c_{i}\alpha_{1,i} \cdots \alpha_{n_{i},i}                     & \longmapsto   & \sum \limits_{i=1}^{l} \beta_{2,i} \cdots \beta_{n_{i},i} ,  
           \end{array}
\end{align*}
where $c_{i,j} \in \CC $, $\alpha_{1, 1} ,\cdots ,\alpha_{n_{k},k}, \beta_{1,1} ,\cdots ,\beta_{n_{l},l}    \in L $. $\sum \limits_{i=1}^{l} \beta_{2,i} \cdots \beta_{n_{i},i}$ is a collection of similar terms of $\sum \limits_{i=1}^{k} \alpha_{2,i} \cdots \alpha_{n_{i},i}$ with the coefficient and the first letter removed.
In particular, $R(c \alpha)=R(c \ew ) =\ew$ for all $c\in \mathcal{R}$, $\alpha \in L$.
We define $R^{k}(x)$ with $k \in \Z_{>0}$, $x \in \RLL$ by
\begin{align*}
R^{k}(x) :=  \overbrace{R(R( \cdots R}^{k}(x) \cdots )).
\end{align*}
In the same way as above, we define the map $\overline{R}$ by 
 \begin{align*}
    \begin{array}{rccc}
\overline{R} \colon &\RLL       &\longrightarrow& \RLL              \\
        & \sum \limits_{i=1}^{k} c_{i}\alpha_{1,i} \cdots \alpha_{n_{i},i}                     & \longmapsto   & \sum \limits_{i=1}^{l} \beta_{1,i} \cdots \beta_{n_{i}-1,i} ,  
           \end{array}
\end{align*}
\en{df}
The map $F$ can be seen as taking the first letter of a word and the map $R$ takes the rest of the word. In contrast, the map $\overline{F}$ can be seen as taking the last letter of a word and the map $\overline{R}$ takes the rest of the word.
Notice that by definition of $F$ and $R$ we have for any word $w\in \RLL$ that $F(w) R(w) = w$ and also  $\overline{R}(w)\overline{F}(w) = w$.
In the classical definition of quasi-shuffle product. (Section \ref{sec:qsh}) we considered a bilinear product $\dia$ on the vector space of letters.
To define the generalized quasi-shuffle product we consider bilinear maps $\diamond:\RL \times \RL \rightarrow \RLL$, i.e. we allow to have words in the image.  

\begin{df}
 We call a $\mathcal{R}$-bilinear map
 \begin{align*}
    \begin{array}{rccc}
\diamond \colon &\mathcal{R} L\times \mathcal{R} L &\longrightarrow&    \mathcal{R}         \langle L \rangle         \\
\end{array}
\end{align*}
\begin{enumerate}[(i)]
    \item \emph{decomposable} if there exists a word $w \in \RLL$ such that for any $\alpha,$ $\beta \in L$ with $R(\alpha \dia \beta) \notin \RL$ we have $w=R(\alpha \dia \beta)$. 
    \item \emph{commutative} if for all $\alpha,\beta \in L$ we have $\alpha \dia \beta = \beta \dia \alpha$\,.
\end{enumerate}

\end{df}
Notice that when $\dia$ is decomposable, we have $F(w\dia v) R(w \dia v) = w \dia v$ for any word $w,v\in \RLL$.


\begin{df}\label{def:generalizedqsh}
For a $\mathcal{R}$-bilinear map $\diamond: \RL \times \RL \rightarrow \RLL$ , we define on $\RLL$ the following $\mathcal{R}$-bilinear product $\gsh$.
For non-empty words $x, y  \in \RLL$, we define $\gsh$ by
\begin{align*}
x \gsh y = F(x) (R(x) \gsh y ) + F(y) (x  \gsh R(y)) + (F(x) \dia F(y))(R(x) \gsh R(y))
\end{align*}
and $\ew \gsh x = x \gsh \ew = x$ for any $x \in \RLL$. We call $\gsh$ the \emph{generalized quasi-shuffle product} associated to $\dia$.
\end{df}
Next, we also define the generalized dual product.

\begin{df}\label{def:rgeneralizedqsh}
For a $\mathcal{R}$-bilinear map $\diamond: \RL \times \RL \rightarrow \RLL$, we define on $\RLL$ the following $\mathcal{R}$-bilinear product $\rgsh$.
For non-empty words $x, y  \in \RLL$, we define $\rgsh$ by
\begin{align*}
x \rgsh y = (\overline{R}(x) \rgsh y )\overline{F}(x)  +  (x  \rgsh \overline{R}(y))\overline{F}(y) + (\overline{R}(x) \rgsh \overline{R}(y))(\overline{F}(x) \dia \overline{F}(y))
\end{align*}
and $\ew \rgsh x = x \rgsh \ew = x$ for any $x \in \RLL$. We call $\rgsh$ the \emph{generalized dual product} associated to $\dia$.
\end{df}

We can prove the next statement in exactly the same way as \cite[Theorem $9$]{Zu2}.
\begin{lem}\label{dualst}
Let $L$ be a letter of set. For a $\mathcal{R}$-bilinear map $\diamond: \RL \times \RL \rightarrow \RLL$, we have,
\begin{align*}
    x\gsh y=x\rgsh y
\end{align*}
for $x,y \in \RLL$.
\end{lem}
\begin{proof}
    We use induction on $l(x)+l(y)$.
    In the case $l(x)+l(y)=0$, immediately we get the statement.
    \begin{enumerate}[(i)]
        \item In the case $x$ or $y$ is an empty word, it is clear by Definition $\ref{def:generalizedqsh}$ and $\ref{def:rgeneralizedqsh}$.
        
        \item In the case $l(x)=1$ and $l(y)>1$,
        there exists $a,b,c \in L$ and $y' \in \RLL$ such that $x=a$ and $y=by'c$.
        Then, we have
        \begin{align*}
        a \gsh by'c &= aby'c+b(a\gsh y'c)+(a\dia b)y'c \\
        &= aby'c+b(a\rgsh y'c)+(a\dia b)y'c \\
        &= aby'c+b(y'ca+(a\rgsh y')c+y'(a\dia c))+(a\dia b)y'c \\
        &= by'ca+(aby'+b(a\rgsh y')+(a\dia b)y')c+by'(a\dia c) \\
        &=by'ca+(aby'+b(a\gsh y')+(a\dia b)y')c+by'(a\dia c) \\
        &=by'ca+(a\gsh by')c+by'(a\dia c) \\
        &=by'ca+(a\rgsh by')c+by'(a\dia c) \\
        &= a \rgsh by'c
        \end{align*}

        \item In the case $l(x)>1$ and $l(y)=1$, this case is similar to (ii).

        \item In the case $l(x),l(y)>1$, there exists $a,b,c,d \in L$ and $x',y' \in \RLL$ such that $x=ax'b$ and $y=cy'd$.
        Then, we have
        \begin{align*}
        ax'b \gsh cy'd &= a(x'b \gsh cy'd)+c(ax'b \gsh y'd)+(a \dia c)(x'b \gsh y'd) \\
        &= a(x'b \rgsh cy'd)+c(ax'b \rgsh y'd)+(a \dia c)(x'b \rgsh y'd) \\
        &= a((x' \rgsh cy'd)b+(x'b \rgsh cy')d+(x' \rgsh cy')(b\dia d)) \\
        &\  \ +c((ax' \rgsh y'd)b+(ax'b \rgsh y')d+(ax' \rgsh y')(b\dia d)) \\
        &\  \ +(a \dia c)((x' \rgsh y'd)b+(x'b \rgsh y')d+(x' \rgsh y')(b\dia d)) \\
        &= a((x' \gsh cy'd)b+(x'b \gsh cy')d+(x' \gsh cy')(b\dia d)) \\
        &\  \ +c((ax' \gsh y'd)b+(ax'b \gsh y')d+(ax' \gsh y')(b\dia d)) \\
        &\  \ +(a \dia c)((x' \gsh y'd)b+(x'b \gsh y')d+(x' \gsh y')(b\dia d)) \\
        &= (a(x'\gsh cy'd)+c(ax' \gsh y'd)+(a \dia c)(x' \gsh y'd))b \\
        & \ \ +(a(x'b \gsh cy')+c(ax'b \gsh y')+(a \dia c)(x'b \gsh y'))d\\
        & \ \ +(a(x' \gsh cy'+c(ax' \gsh y')+(a \dia c)(x' \gsh y'))(b\dia d) \\
        &= (ax'\gsh cy'd)b+(ax'b\gsh cy')d+(ax'\gsh cy')(b\dia d) \\
        &= (ax'\rgsh cy'd)b+(ax'b\rgsh cy')d+(ax'\rgsh cy')(b\dia d) \\
        &=ax'b \rgsh cy'd
\end{align*}
    \end{enumerate}
\end{proof}
In the following our goal will be to understand when the $\gsh$ is associative. For this we will need to prove some lemmas.

\begin{lem}
\label{l1}
For a commutative $\mathcal{R}$-bilinear map $\diamond: \RL \times \RL \rightarrow \RLL$, the following two statements are equivalent.
\begin{enumerate}[(i)]
\item
For $\alpha , \beta \in L$, $x,y \in \RLL$, we have
\begin{align*}
(R^{k}(\alpha \dia \beta)x\gsh y) = F(R^{k}(\alpha \dia \beta)x)(R^{k+1}(\alpha \dia \beta)x \gsh y)
\end{align*}
for any $k \in \Z$ with $1 \ika k < l(\alpha,\beta)$.
\item
For all $\alpha'$, $\beta' \in L$ with $l(\alpha', \beta')\ijo 2$, we have
\begin{align}
\label{joken1}
       \gamma \dia  F(R^{k'}(\alpha' \dia \beta')) = -\gamma  F(R^{k'}(\alpha' \dia \beta'))
\end{align}
for any $\gamma \in L$, and $k' \in \Z$ with $1 \ika k' < l(\alpha' , \beta')$.
\end{enumerate}
\end{lem}

\begin{proof}
$(\Leftarrow)$
\begin{enumerate}[(I)]
\item In the case $l(\alpha, \beta)=0,1$, we don't have to consider these cases since a condition which is $1 \ika k < l(\alpha,\beta)$ of (i).
\item In the case $l(\alpha, \beta) \ijo 2$, we use induction on $l(x)+l(y)$.\\
In the case $l(x)+l(y)=0$, we can get the statement immediately.
In the case $l(x)+l(y)=n$, we calculate directly
\begin{align*}
(R^{k}(\alpha \dia \beta)x) \gsh y = &F(R^{k}(\alpha \dia \beta)x) (R(R^{k}(\alpha \dia \beta)x) \gsh y ) \color{black}+ F(y)( (R^{k}(\alpha \dia \beta)x) \gsh R(y)) \\
&\color{black}+ ( F(R^{k}(\alpha \dia \beta)x) \dia F(y))(R(R^{k}(\alpha \dia \beta)x) \gsh R(y) ).
\end{align*}
By the induction hypothesis,
\begin{align*}
F(\alpha \dia \beta) &\seki_{j=1}^{k-1} F(R^{j}(\alpha \dia \beta)x)(R^{k}(\alpha \dia \beta)x\gsh R(y))\\
&= F(\alpha \dia \beta)\seki_{j=1}^{k}F(R^{j}(\alpha \dia \beta)x)(R^{k+1}(\alpha \dia \beta)x \gsh R(y)). \\
 \Leftrightarrow R^{k}(\alpha \dia \beta)x\gsh R(y) &= F(R^{k}(\alpha \dia \beta)x)(R^{k+1}(\alpha \dia \beta)x \gsh R(y)). 
\end{align*}
Then,
\begin{align*}
&\color{black}F(y)( (R^{k}(\alpha \dia \beta)x) \gsh R(y)) + ( F(R^{k}(\alpha \dia \beta)x) \dia F(y))(R(R^{k}(\alpha \dia \beta)x) \gsh R(y) ) \\
= F(y) &F(R^{k}(\alpha \dia \beta)x)(R^{k+1}(\alpha \dia \beta)x \gsh R(y))\\
 &+  ( F(R^{k}(\alpha \dia \beta)x) \dia F(y))(R(R^{k}(\alpha \dia \beta)x) \gsh R(y) ) \cdots (**)
\end{align*}

Now we assume that in the case $l(\alpha ,\beta) \ijo 2$,
\begin{align*}
       \gamma \dia  F(R(\alpha \dia \beta)) = -\gamma  F(R(\alpha \dia \beta)) \\
\end{align*}
for any $\gamma \in L $. In particular, for $\gamma=F(y)$ we get
\begin{align*}
(**) &= \big( F(y) F (R (\alpha \dia \beta) x) + F(R(\alpha \dia \beta)x) \dia F(y) \big)(R (R (\alpha \dia \beta )x) \gsh R(y))  \\
&= 0.
\end{align*}
Then, we have $$(R^{k}(\alpha \dia \beta)x\gsh y) = F(R^{k}(\alpha \dia \beta)x)(R^{k+1}(\alpha \dia \beta)x \gsh y).$$
\end{enumerate}

$(\Rightarrow)$
We assume that there exist  $\alpha$, $\beta \in L$ such that $l(\alpha, \beta)\ijo 2$.
For $k\in \Z$ with $1 \ika k < l(\alpha,\beta)$,
we calculate directly
\begin{align*}
(R^{k}(\alpha \dia \beta)x) \gsh y = &F(R^{k}(\alpha \dia \beta)x) (R(R^{k}(\alpha \dia \beta)x) \gsh y ) +\color{black} F(y)( (R^{k}(\alpha \dia \beta)x) \gsh R(y)) \\
&\color{black}+ ( F(R^{k}(\alpha \dia \beta)x) \dia F(y))(R(R^{k}(\alpha \dia \beta)x) \gsh R(y) ).
\end{align*}
By assumption of (i), we get 
\begin{align*}
 (R^{k}(\alpha \dia \beta)x\gsh y) = F(R^{k}(\alpha \dia \beta)x)(R^{k+1}(\alpha \dia \beta)x \gsh y).
\end{align*}
Therefore, we can get
\begin{align*}
&\color{black}F(y)( (R^{k}(\alpha \dia \beta)x) \gsh R(y)) + ( F(R^{k}(\alpha \dia \beta)x) \dia F(y))(R(R^{k}(\alpha \dia \beta)x) \gsh R(y) ) \\
&= F(y) F (R^{k} (\alpha \dia \beta) x) (R (R^{k} (\alpha \dia \beta )x) \gsh R(y)) +  ( F(R^{k}(\alpha \dia \beta)x) \dia F(y))(R(R^{k}(\alpha \dia \beta)x) \gsh R(y) ) \\
&= ( F(y) F (R^{k} (\alpha \dia \beta) x) +( F(R^{k}(\alpha \dia \beta)x) \dia F(y)) )(R (R^{k} (\alpha \dia \beta )x) \gsh R(y))= 0. \\
&\Leftrightarrow	F(y) F (R^{k} (\alpha \dia \beta) ) +( F(R^{k}(\alpha \dia \beta)) \dia F(y)) = 0.
\end{align*}
Then, we can get
$$ \gamma \dia  F(R(\alpha \dia \beta)) = -\gamma  F(R(\alpha \dia \beta)) $$
when $l(\alpha , \beta) \ijo 2$. 
\end{proof}

\begin{prop}
\label{p1}
Let $\alpha, \beta \in L$ be letters and the commutative and decomposable $\mathcal{R}$-bilinear map $\diamond: \RL \times \RL \rightarrow \RLL$ satisfy $(ii)$ in Lemma \ref{l1}.
Assume that $l(\alpha \dia \beta) \ijo 2$ and $\gamma_{1} , \gamma_{2} \in L \backslash F(R(\alpha \dia \beta))$.
Then, we have \\
$$\gamma_{1} \dia \gamma_{2} \in \RL  \  \quad\text{or}\quad \  F(R(\alpha \dia \beta))=F(R(\gamma_1 \dia \gamma_2)).$$
\end{prop}

\begin{proof}
We assume that $l(\gamma_1 \dia \gamma_2) \ijo 2 $ with $\gamma_1$, $\gamma_2 \in L \backslash F(R(\alpha \dia \beta))$.
By the assumption of $\dia$, we have 
$$F(R(\gamma_1 \dia \gamma_2)) \dia x = x \dia F(R(\gamma_1 \dia \gamma_2)) = -xF(R(\gamma_1 \dia \gamma_2))$$
for $x \in L$.
We put $x=F(R(\alpha \dia \beta))$. Then,
\begin{align}
\label{f1}
F(R(\gamma_1 \dia \gamma_2)) \dia F(R(\alpha \dia \beta) = F(R(\alpha \dia \beta) \dia F(R(\gamma_1 \dia \gamma_2)) = -F(R(\alpha \dia \beta)F(R(\gamma_1 \dia \gamma_2)).
\end{align}
Furthermore by $l(\alpha \dia \beta) \ijo 2$, for $y \in L$ we have
$$F(R(\alpha_1 \dia \beta_2)) \dia y = y \dia F(R(\alpha_1 \dia \beta_2)) = -yF(R(\alpha_1 \dia \beta_2)).$$
Then we put $y=F(R(\gamma_1 \dia \gamma_2))$. We have
\begin{align}
\label{f2}
F(R(\alpha_1 \dia \beta_2)) \dia F(R(\gamma_1 \dia \gamma_2)) = F(R(\gamma_1 \dia \gamma_2)) \dia F(R(\alpha_1 \dia \beta_2)) = -F(R(\gamma_1 \dia \gamma_2))F(R(\alpha_1 \dia \beta_2)).
\end{align}
Comparing (\ref{f1}) and (\ref{f2}) we get
$F(R(\alpha \dia \beta))=F(R(\gamma_1 \dia \gamma_2))$.
\end{proof}

\begin{prop}\label{p2}
Let $\alpha, \beta$ be a letter in $L$ and let $k=l(\alpha,\beta)$. Assume the commutative and decomposable $\mathcal{R}$-bilinear map $\diamond: \RL \times \RL \rightarrow \RLL$ satisfies (ii) in Lemma \ref{l1}. Then we have
\begin{align*}
R(\alpha \dia \beta) = \overbrace{F(R(\alpha \dia \beta)) \cdots   F(R(\alpha \dia \beta))}^{k-1}.
 \end{align*}
\end{prop}
\begin{proof}
\begin{enumerate}[(i)]
\item In the case $l(\alpha \dia \beta)= 1,2$, we can get the statement immediately.
\item In the case $l(\alpha \dia \beta)> 2$, we have 
by Lemma $\ref{l1}$
\begin{align*}
\gamma \dia F(R^{k}(\alpha\dia\beta))=-\gamma F(R^{k}(\alpha \dia \beta))\\
\gamma \dia F(R^{l}(\alpha\dia\beta))=-\gamma F(R^{l}(\alpha \dia \beta))
\end{align*}
for all $\gamma \in L$ and $n,m \in \Z$ with $0 \ika n,m \ika l(\alpha,\beta)$ and $n \neq m$.
Then, we can get the statement by comparing the case in which we put  $\gamma=F(R^{k}(\alpha \dia \beta))$ and the case in which we put $\gamma=F(R^{l}(\alpha \dia \beta))$.
\end{enumerate}
\end{proof}

\begin{lem}
\label{l2}
Let $\alpha, \beta , \gamma$ be in $L$ and $|L| \ika 2$.
For a commutative and decomposable $\mathcal{R}$-bilinear map $\diamond: \RL \times \RL \rightarrow \RLL$, we have
\begin{align*}
(F(\alpha \dia \beta)\dia \gamma)R(\alpha \dia \beta)=(\alpha \dia F(\beta \dia \gamma))R(\beta \dia \gamma)
\end{align*}
whenever 
\begin{align}
\label{joken2}
\begin{cases}
F(F(\alpha \dia \beta)\dia \gamma)=F(\alpha \dia F(\beta \dia \gamma)), &\mathrm{if }\max \{l(\alpha , \beta) \mid \alpha, \beta \in L\} =1. \\
\dia \ \mathrm{satisfies} \  \mathrm{(ii)}\ \mathrm{in}\ \mathrm{Lemma} \ \ref{l1} \  ,&\mathrm{if }\max \{l(\alpha , \beta) \mid \alpha, \beta \in L  \}  >1.
\end{cases}
\end{align}

\end{lem}
\begin{proof}
\begin{enumerate}[(I)]
\item In the case $l(\alpha, \beta)=0,1$ for $\alpha, \beta \in L$.\\
Immediately, we can get the statement since $F(F(\alpha \dia \beta)\dia \gamma)=F(\alpha \dia F(\beta \dia \gamma))$.
\item In the case $\max \{l(\alpha , \beta) \mid \alpha, \beta \in L  \}  >1$, there exists $x, y \in L$ such that $l(x \dia y) \ijo 2$.\\
We put $M=F(R(x \dia y))$.
For all $\alpha, \beta , \gamma \in L$, we check $$(F(\alpha \dia \beta)\dia \gamma)R(\alpha \dia \beta)=(\alpha \dia F(\beta \dia \gamma))R(\beta \dia \gamma)$$
for different cases separately. 
\begin{enumerate}
\item In the case $\beta =M$ we get by Lemma $\ref{l1}$,
\begin{align*}
\alpha \dia \beta = - \alpha M.
\end{align*}
Then, we have
\begin{align*}
(F(\alpha \dia \beta)\dia \gamma)R(\alpha \dia \beta) = (- \alpha \dia \gamma) M \,.
\end{align*}
By Lemma $\ref{l1}$,
\begin{align*}
\beta \dia \gamma = - \gamma M\,.
\end{align*}
Then, we have
\begin{align*}
(\alpha \dia F(\beta \dia \gamma))R(\beta \dia \gamma)= (- \alpha \dia \gamma) M \,.
\end{align*}
Therefore,
\begin{align*}
(F(\alpha \dia \beta)\dia \gamma)R(\alpha \dia \beta)=(\alpha \dia F(\beta \dia \gamma))R(\beta \dia \gamma)\,.
\end{align*}

\item The case $\beta \neq M$:
\begin{enumerate}[(i)]
\item $\alpha=\gamma=M$,
By Lemma $\ref{l1}$,
\begin{align*}
\alpha \dia \beta = - \alpha M \,.
\end{align*}
Then, we have
\begin{align*}
(F(\alpha \dia \beta)\dia \gamma)R(\alpha \dia \beta) = (- \alpha \dia \gamma) M \,.
\end{align*}
By Lemma $\ref{l1}$,
\begin{align*}
\beta \dia \gamma = - \gamma M\,.
\end{align*}
Then, we have
\begin{align*}
(\alpha \dia F(\beta \dia \gamma))R(\beta \dia \gamma)= (- \alpha \dia \gamma) M \,.
\end{align*}
Therefore,
\begin{align*}
(F(\alpha \dia \beta)\dia \gamma)R(\alpha \dia \beta)=(\alpha \dia F(\beta \dia \gamma))R(\beta \dia \gamma)\,.
\end{align*}

\item $\alpha=M$, $\gamma \neq M$,
by Proposition $\ref{p1}$, $\ref{p2}$, there exists $A \in L$ and $n \in \Z_{\ijo 0}$ such that $\beta \dia \gamma = - A M^{n}$.
By Lemma $\ref{l1}$, we have
\begin{align*}
(\alpha \dia F(\beta \dia \gamma))R(\beta \dia \gamma)= ( \alpha \dia -A) M^{n} = ( M \dia -A) M^{n}=AM^{n+1}.
\end{align*}
By Lemma $\ref{l1}$,
\begin{align*}
\alpha \dia \beta = - \beta M\,.
\end{align*}
Then, we have
\begin{align*}
(F(\alpha \dia \beta)\dia \gamma)R(\alpha \dia \beta) = (- \beta \dia \gamma) M = AM^{n+1}\,.
\end{align*}
Therefore,
\begin{align*}
(F(\alpha \dia \beta)\dia \gamma)R(\alpha \dia \beta)=(\alpha \dia F(\beta \dia \gamma))R(\beta \dia \gamma).
\end{align*}

\item $\alpha \neq M$, $\gamma = M$, we can get the statement since $\dia$ is commutative and (ii).

\item $\alpha \neq M$, $\gamma \neq M$,
This case is all same letters since $|L|\ika 2$.( i.e. $\alpha=\beta=\gamma$.)
By Proposition $\ref{p1}$, $\ref{p2}$, there exist $C \in L$ and $p \in \Z_{\ijo 0}$ such that $\alpha \dia \beta = - C M^{p}$, i.e.
\begin{align*}
(F(\alpha \dia \alpha)\dia \alpha)R(\alpha \dia \alpha) = (- C \dia \alpha) M^{p}, \\
(\alpha \dia F(\alpha \dia \alpha))R(\alpha \dia \alpha)=(\alpha \dia -C)M^{p}.
\end{align*}
Therefore,
\begin{align*}
(F(\alpha \dia \beta)\dia \gamma)R(\alpha \dia \beta)=(\alpha \dia F(\beta \dia \gamma))R(\beta \dia \gamma).
\end{align*}
\end{enumerate}
\end{enumerate}
\end{enumerate}
\end{proof}

\begin{df}
We call a decomposable $\mathcal{R}$-bilinear map $\diamond: \RL \times \RL \rightarrow \RLL$ \emph{associative} if for any $\alpha,\beta,\gamma \in L$
\begin{align*}
    (F(\alpha \dia \beta)\dia \gamma)R(\alpha \dia \beta)=(\alpha \dia F(\beta \dia \gamma))R(\beta \dia \gamma)
\end{align*}
and for  $\alpha,\beta \in L$ with $l(\alpha,\beta) \geq 2$, we have for $\gamma \in L $
\begin{align*}
    \gamma \dia F(R(\alpha \dia \beta))= 
        -\gamma F(R(\alpha \dia \beta)).
\end{align*} 
\end{df}

Notice that if $\alpha \diamond \beta \in \mathcal{R}L$ for all $\alpha, \beta \in L$, then $\diamond$ being associative just means $$\alpha\diamond (\beta \diamond \gamma) = (\alpha\diamond \beta) \diamond \gamma$$ for all $\alpha,\beta,\gamma \in L$.

\begin{thm}
\label{main}
If $\diamond$ is associative
then the associated generalized quasi-shuffle product $\gsh$ is associative.
\end{thm}

\begin{proof}
Let $\alpha, \beta, \gamma \in  \RLL$.
We need to check $\alpha \gsh (\beta \gsh \gamma)= (\alpha \gsh \beta)\gsh \gamma$.
We use induction on $l(\alpha)+l(\beta)+l(\gamma)$.\\
In the case $l(\alpha)+l(\beta)+l(\gamma)=1$, immediately we can get the statement.\\
In the case $l(\alpha)+l(\beta)+l(\gamma)=k$,
\begin{align*}
&\alpha \gsh (\beta \gsh \gamma) \\
&= \alpha \gsh (F(\beta)(R(\beta)\gsh \gamma)+F(\gamma)(\beta \gsh R(\gamma))+(F(\beta)\dia F(\gamma))(R(\beta)\gsh R(\gamma)))\\
&= \color{black}F(\alpha)(R(\alpha)\gsh  F(\beta)(R(\beta)\gsh \gamma)) 
\color{black}+F(\beta)(\alpha \gsh (R(\beta) \gsh \gamma))
\color{black}+(F(\alpha) \dia F(\beta))(R(\alpha)\gsh (R(\beta) \gsh \gamma))\\
&\color{black}+F(\alpha)(R(\alpha)\gsh F(\gamma)(\beta \gsh R(\gamma)))
\color{black}+F(\gamma)(\alpha \gsh (\beta \gsh R(\gamma)))
\color{black}+(F(\alpha) \dia F(\gamma))(R(\alpha)\gsh (\beta \gsh R(\gamma))) \\
&\color{black}+F(\alpha)(R(\alpha)\gsh (F(\beta)\dia F(\gamma))(R(\beta)\gsh R(\gamma)))
\color{black}+F(F(\beta)\dia F(\gamma))(\alpha \gsh R(F(\beta) \dia F(\gamma))(R(\beta)\gsh R(\gamma))) \\
&\color{black}+(F(\alpha)\dia F(F(\beta)\dia F(\gamma)))(R(\alpha) \gsh R((F(\beta)\dia F(\gamma))(R(\beta)\gsh R(\gamma)))).
\end{align*}
\begin{align*}
&(\alpha \gsh \beta)\gsh \gamma \\
&= 
(F(\alpha)(R(\alpha)\gsh \beta)+F(\beta)(\alpha \gsh R(\beta))+(F(\alpha \dia \beta)(R(\alpha)\gsh R(\beta))))\gsh \gamma \\
&= \color{black}F(\alpha)((R(\alpha)\gsh \beta)\gsh \gamma) 
+ \color{black}F(\gamma)(F(\alpha)(R(\alpha)\gsh \beta)\gsh R(\gamma))
+ \color{black}(F(\alpha)\dia F(\gamma))((R(\alpha)\gsh \beta)\gsh R(\gamma)) \\
&+\color{black}F(\beta)((\alpha \gsh R(\beta))\gsh \gamma)
+\color{black}F(\gamma)(F(\beta)(\alpha \gsh R(\beta))\gsh R(\gamma))
\color{black}+(F(\beta)\dia F(\gamma))((\alpha \gsh R(\beta))\gsh R(\gamma)) \\
&\color{black}+F(F(\alpha)\dia F(\beta))(R(F(\alpha)\dia F(\beta))((R(\alpha)\gsh R(\beta))\gsh \gamma))
+\color{black}F(\gamma)((F(\alpha)\dia F(\beta))(R(\alpha)\gsh R(\beta))\gsh R(\gamma)) \\
&\color{black}+(F(F(\alpha) \dia F(\beta))\dia F(\gamma))(R(F(\alpha) \dia F(\beta))(R(\alpha)\gsh R(\beta))\gsh R(\gamma))\,.
\end{align*}
Immediately, we can get the satisfies by an induction hypothesis. 
\begin{align*}
\color{black}F(\beta)(\alpha \gsh (R(\beta) \gsh \gamma)) &= \color{black}F(\beta)((\alpha \gsh R(\beta))\gsh \gamma), \\
\color{black}(F(\alpha) \dia F(\gamma))(R(\alpha)\gsh (\beta \gsh R(\gamma))) &= \color{black}(F(\alpha)\dia F(\gamma))((R(\alpha)\gsh \beta)\gsh R(\gamma)) \, .
\end{align*}
\begin{align*}
\color{black}F(\gamma)(\alpha \gsh (\beta \gsh R(\gamma))) &= \color{black}F(\gamma)((\alpha \gsh \beta) \gsh R(\gamma)) \\
&= \color{black}F(\gamma)(F(\alpha)(R(\alpha)\gsh \beta)\gsh R(\gamma))+\color{black}F(\gamma)(F(\beta)(\alpha \gsh R(\beta))\gsh R(\gamma)) \\
&+\color{black}F(\gamma)((F(\alpha)\dia F(\beta))(R(\alpha)\gsh R(\beta))\gsh R(\gamma))\, .
\end{align*}
\begin{align*}
    \color{black}F(\alpha)((R(\alpha)\gsh \beta)\gsh \gamma) &= \color{black}F(\alpha)(R(\alpha)\gsh ( \beta\gsh \gamma)) \\
    &= \color{black}F(\alpha)((R(\alpha)\gsh \beta)\gsh \gamma) + \color{black}F(\alpha)(R(\alpha)\gsh F(\gamma)(\beta \gsh R(\gamma))) \\
    \color{black}
    &+F(\alpha)(R(\alpha)\gsh (F(\beta)\dia F(\gamma))(R(\beta)\gsh R(\gamma)))\,.
\end{align*}
By $\dia$ satisfies decomposable, we have
\begin{align*}
F(\alpha) \dia F(\beta) = F(F(\alpha)\dia F(\beta))R(F(\alpha)\dia F(\beta))\, .
\end{align*}

By Lemma $\ref{l1}$
\begin{align*}
    \color{black}(F(\alpha) \dia F(\beta))(R(\alpha)\gsh (R(\beta) \gsh \gamma))=\color{black}F(F(\alpha)\dia F(\beta))(R(F(\alpha)\dia F(\beta))((R(\alpha)\gsh R(\beta))\gsh \gamma))\,.
\end{align*}
\begin{align*}
    \color{black}F(F(\beta)\dia F(\gamma))(\alpha \gsh R(F(\beta) \dia F(\gamma))(R(\beta)\gsh R(\gamma)))= \color{black}(F(\beta)\dia F(\gamma))((\alpha \gsh R(\beta))\gsh R(\gamma))\,.
\end{align*}
By assumption of $\dia$ and Lemma $\ref{l1}$, we have
\begin{align*}
    &\color{black}(F(\alpha)\dia F(F(\beta)\dia F(\gamma)))(R(\alpha) \gsh R((F(\beta)\dia F(\gamma))(R(\beta)\gsh R(\gamma)))) \\
    &=\color{black}(F(F(\alpha) \dia F(\beta))\dia F(\gamma))(R(F(\alpha) \dia F(\beta))(R(\alpha)\gsh R(\beta))\gsh R(\gamma))\,.
\end{align*}
\end{proof}
 
The following can be seen as a generalization of Theorem \ref{thm:quasishufflealgebra}.

\begin{thm}\label{thm:main2}
If $\diamond$ is commutative and associative, then the space $\mathcal{R}\langle L \rangle$ equipped with $\gsh$ becomes a commutative $\mathcal{R}$-algebra.
\end{thm}

\begin{ex}\label{ex:3}
 Let $L = \{ a, b \}$ and $\mathcal{R}=\mathcal{C}$ define $\dia$ by
 \begin{align*}
    \begin{array}{rccc}
\dia \colon &\mathcal{C} L\times\mathcal{C} L &\longrightarrow&    \mathcal{C}         \langle L \rangle         \\
        & (a,b)                   & \longmapsto   & -ab  \\
        & (b,a)                   & \longmapsto   & -ab  \\
        & (b,b)                   & \longmapsto   & -bb  \\
        & (a,a)                   & \longmapsto   & -abb  \\
           \end{array}
\end{align*}
First, we check this $\dia$ is associative. 
Immediately, we can get the following.
For $\alpha,\beta \in L$ with $l(\alpha,\beta) \geq 2$, we have for $\gamma \in L $
\begin{align*}
    \gamma \dia F(R(\alpha \dia \beta))= 
        -\gamma F(R(\alpha \dia \beta)).
\end{align*} 
Since Lemma $\ref{l2}$, we have
\begin{align*}
    (F(\alpha \dia \beta)\dia \gamma)R(\alpha \dia \beta)=(\alpha \dia F(\beta \dia \gamma)R(\beta \dia \gamma))
\end{align*}
for $\alpha,\beta,\gamma \in L$.
Then, we can use Theorem $\ref{main}$.
For example, we have
\begin{align*}
(a \gsh b) \gsh b &= (ab+ba-ab)\gsh b\\
&= ba \gsh b = b(a \gsh b)+ bba -bba = bba,\\[5pt]
a \gsh( b \gsh b) &= a \gsh (bb + bb -bb)=a \gsh  bb \\
&= abb + b(a \gsh  b)-abb = bba.
\end{align*}
Therefore we have $(a \gsh b) \gsh b = a \gsh( b \gsh b) .$ Also, since
\begin{align*}
(a \gsh a) \gsh b &= (aa+aa-abb)\gsh b\\
&= (2aa-abb) \gsh b = 2a(a \gsh b)+ 2bba -2aba-(a(bb \gsh b)+babb-abbb) \\
&= 2aba+2baa-2aba-(a(b(b \gsh b)+bbb-bbb)+babb-abbb) \\
&= 2baa - babb,\\[5pt]
a \gsh( a \gsh b) &=a \gsh  ab 
= aba + b(a \gsh  a)-aba = b(2aa-abb)=2baa-babb,
\end{align*}
we get $(a \gsh a) \gsh b = a \gsh( a \gsh b) $.
\end{ex}

Next, we give another example often called \emph{integral shuffle product} using the generalized quasi-shuffle product. We can see the integral shuffle product in \cite{T0} 
\begin{ex}[\cite{T0} Section 2.3]\label{ex:4}
Let $L = \{ a, b, c \}$ and $\mathcal{R}=\mathcal{C}$ define $\dia$ by
 \begin{align*}
    \begin{array}{rccc}
\dia \colon &\mathcal{C} L\times\mathcal{C} L &\longrightarrow&    \mathcal{C}         \langle L \rangle         \\
        & (a,b),(b,a)                  & \longmapsto   & 0  \\
        & (b,b)                   & \longmapsto   & -bc  \\
        & (a,a)                   & \longmapsto   & \hbar a  \\
        & (c,a), (a,c)                   & \longmapsto   & -ac  \\
        & (c,b), (b,c)                   & \longmapsto   & -bc  \\
        & (c,c)                   & \longmapsto   & -cc  \\
           \end{array}
\end{align*}
Also here we see that this $\dia$ is associative. For example,
\begin{align*}
    (F(a\dia c)\dia b)R(a\dia c)&=(-a\dia b)c=0 \\
    (a\dia F(c\dia b))R(c\dia b)&=(a \dia -b)c=0.\\
    \\
    (F(a\dia a)\dia c)R(a\dia a)&=\hbar a\dia c=-\hbar ac \\
    (a\dia F(a\dia c))R(a\dia c)&=(a \dia -a)c=-\hbar ac.\\
\end{align*}
Therefore the associated $\gsh$ is associative.
For example, we have
\begin{align*}
    (a \gsh b)\gsh c &= (ab+ba)\gsh c \\
    &= cab+a(b \gsh c)-acb+cba+b(a \gsh c)-bca\\
    &=cab+a(bc+cb-bc)-acb+cba+b(ac+ca-ac)-bca\\
    &=cab+acb-acb+cba+bca-bca\\
    &=cab+cba.\\
    \\
    a \gsh (b \gsh c)&= a \gsh (bc + cb -bc) \\
    &= a \gsh cb \\
    &= acb + c(a \gsh b)-acb \\
    &= c(ab+ba)\\
    &= cab+cba.
\end{align*}
\end{ex}
\subsection{Shuffle product for q-analogues}
\label{Shuffle:q}
In this section, we introduce the $q$-shuffle product as an example of the generalized quasi-shuffle product and check that $\zeta_q$ is a $\CC$-algebra homomorphism with respect to the $q$-shuffle product. After this, we show how the $q$-shuffle product can be expressed in terms of the $q$-stuffle product. \\
First, we define $q$-shuffle product using the generalized quasi-shuffle product.
\begin{df}\label{def:qshuffle}
\begin{enumerate}[(i)]
    \item 
For $L=\{a,b\}$ we define $\dia_{q}$ by
 \begin{align*}
    \begin{array}{rccc}
\dia_{q} \colon &\mathcal{C} L\times\mathcal{C} L &\longrightarrow&    \mathcal{C}         \langle L \rangle         \\
        & (a,b)                   & \longmapsto   & -ab  \\
        & (b,a)                   & \longmapsto   & -ab  \\
        & (b,b)                   & \longmapsto   & -bb  \\
        & (a,a)                   & \longmapsto   & \hbar a  \\
           \end{array}
\end{align*}
\item We define $\shuffle_{q}$ as the generalized quasi-shuffle product associated to $\diamond=\dia_{q}$. 
\end{enumerate}
\end{df}

It is easy to see that $\diamond_q$ is commutative and associative and therefore, by Theorem \ref{thm:main2}, the space $\qgerH$ equipped with $\shuffle_q$ becomes a commutative $\mathcal{C}$-algebra. Notice that $\qgerH^0$ and $\qgerH^1$ are both closed under $\shuffle_q$, i.e. these are subalgebras of $\qgerH$.
 
\begin{prop}
\label{ex12}
For all $f,g \in z\Q [[ q,z ]]$ and $\alpha , \beta \in L$, we have 
\begin{align*}
( \alpha f)( \beta g)= \alpha f (\beta g) + \beta (\alpha f )g + (\alpha \dia_q \beta )(f g).
\end{align*}
\end{prop}
\begin{proof}
We check the statement by direct calculation.
\begin{enumerate}[(i)]
\item $\alpha = \beta = a$
\begin{align*} 
(af)(ag) &=(1-q)\sum_{i=1}^{\infty} f|_{z=q^iz} (1-q)\sum_{j=1}^{\infty} g|_{z=q^jz} = (1-q)^{2} \sum_{i,j=1}^{\infty} f|_{z=q^jz} g|_{z=q^jz} \\
&= (1-q)^{2} ( \sum_{i>j>0} f|_{z=q^iz}  g|_{z=q^jz} + \sum_{j>i>0} f|_{z=q^iz}  g|_{z=q^jz} + \sum_{i=1}^{\infty} f|_{z=q^iz} g|_{z=q^iz} ) \\
&= (1-q)^{2} ( \sum_{r=1}^{\infty} \sum_{j=1}^{\infty} f|_{z=q^{j+r}z}  \ g|_{z=q^jz} + \sum_{d=1}^{\infty}\sum_{i=1}^{\infty} f|_{z=q^iz}  \ g|_{z=q^{i+d}z} + \sum_{i=1}^{\infty} f|_{z=q^jz} \ g|_{z=q^jz} ) \\
&= a ( (1-q) \sum_{j=1}^{\infty} f|_{z=q^{j}z} \ g) + a ( (1-q)  f \sum_{i=1}^{\infty}   \ g|_{z=q^{i}z}) + a ( (1-q) fg) \\
&= a(af)g+af(ag)+(1-q)a fg \\
&= \alpha f (\beta g) + \beta (\alpha f )g + (\alpha \dia \beta )(f g).
\end{align*}

\item $\alpha = a$, $\beta = b$,
\begin{align*}
a \cdot f (b \cdot g) &= a \cdot \left( f \frac{z}{1-z} g \right). \\
(a \dia b) \cdot (fg) &= -ab \cdot (fg) = -a \cdot( \frac{z}{1-z} fg ) = -a \cdot \left( f \frac{z}{1-z} g \right).
\end{align*}
From this we get
$$( \alpha f)( \beta g)= \alpha f (\beta g) + \beta (\alpha f )g + (\alpha \dia \beta )(f g).$$

\item $\alpha = b$, $\beta = b$,
\begin{align*}
b \cdot f (b \cdot g) &= b \cdot \left( f \frac{z}{1-z} g \right). \\
(b \dia b) \cdot (fg) &= -bb \cdot (fg) = -b \cdot( \frac{z}{1-z} fg ) = -b \cdot \left( f \frac{z}{1-z} g \right).
\end{align*}
From this we get
$$( \alpha f)( \beta g)= \alpha f (\beta g) + \beta (\alpha f )g + (\alpha \dia \beta )(f g).$$

\end{enumerate}
\end{proof}

 \begin{thm}
 \label{Liq:is:hom}
 The map $\Li^{q}$ is a $\mathcal{C}$-algebra homomorphisms.
 \end{thm}
\begin{proof}
We use an induction step on $l(w)+l(v)$. In the case $l(w)+l(v)=1$, immediately we can get the statement. \\
In the case $l(w)+l(v)=2$, the statement follows from Proposition \ref{ex11}.
We assume that the statement holds when $l(w)+l(v)=n-1$.
In the case $l(w)+l(v)=n$,
 \item we put $w=\alpha w'$ and $v= \beta v'$ $(\alpha , \beta \in \{ a, b\})$
since  Proposition $\ref{ex11}$, $\ref{ex12}$,
 \begin{align*}
 \Li^{q}_{\alpha w'}(z) \Li^{q}_{bv'}(z) &= \alpha \Li^{q}_{w'}(z) (\beta \Li^{q}_{v'}(z)) + \beta (\alpha \Li^{q}_{w'}(z))\Li^{q}_{v'}(z) + (\alpha \dia \beta )(\Li^{q}_{w'}(z) \Li^{q}_{v'}(z)) \\
 &= \alpha \Li^{q}_{w' \shuffle_{q} \beta v' }(z) + \beta \Li^{q}_{ \alpha w' \shuffle_{q} v'}(z) + \Li^{q}_{(\alpha \dia \beta )w' \shuffle v'}(z) \\
 &= \Li^{q}_{\alpha (w' \shuffle_{q} \beta v') }(z) +  \Li^{q}_{ \beta (\alpha w' \shuffle_{q} v')}(z) + \Li^{q}_{(\alpha \dia \beta )(w' \shuffle v')}(z) \\
 &= \Li^{q}_{\alpha w' \shuffle_{q} \beta v'} (z).
 \end{align*}
\end{proof}

\begin{coro}
For $w$, $u \in \qgerH^{0}$ we have
\begin{align*}
\zeta_{q}(w)\zeta_{q}(u)=\zeta_{q}(w\shuffle_q u).
\end{align*}
In particular, the map $\zeta_q$ is a $\mathcal{C}$-algebra homomorphism from $\qgerH^{0}$ to $\Q[[q]]$. 
\end{coro}
\begin{proof}
We can get the statement by Theorem $\ref{Liq:is:hom}$ and $z \rightarrow 1$.
\end{proof}

\subsection{The dual stuffle product for multiple zeta values}\label{dual:stuffle}

In \cite{EMS} and \cite{Zu} the authors consider a dual stuffle product defined for $w,v \in \gerH^0$ by $$\tau(\tau(w) \ast \tau(v)),$$ where $\tau$ denotes the classical duality defined in Definition \ref{def:tau} (In \cite{EMS} this product is denoted $m_\square$). In this section, we will show that the dual stuffle product can be described by using our notion of generalized quasi-shuffle product. For this, we first give the following more general version of Proposition $\ref{q_dual}$ from the previous section.
\begin{prop}\label{generalized_dual}
Let $L=\{a,b\}$ and $E=\{ e_{0},e_{1},e_{2}, \dots \}$.
Assume that for $k,l \geq 0$ the commutative $\mathcal{C}$-bilinear map $\diamond_{E}: \CL \times \CL \rightarrow \CLL$ of fixed length is given by
\begin{align*}
e_{k} \dia_{E} e_{l} = a^{k+l}(e_{0} \dia_{E} e_{0} ),
\end{align*}
for some choice of $e_{0} \dia_{E} e_{0}$.
We identify $e_k = a^k b \in \qgerH$ and assume that the commutative $\mathcal{C}$-bilinear map $\diamond_{L}: \CL \times \CL \rightarrow \CLL$ of fixed length is given by
\begin{enumerate}[(i)]
    \item $a \dia_{L} a = \hbar^{2}\sigma (e_{0} \dia_{E} e_{0})$.
    \item $a \dia_{L} b = b \dia_{L} a = -ab$.
    \item $b \dia_{L} b = -bb$.
\end{enumerate}
Denote by $\shuffle_L$ and $\shuffle_E$ the associated generalized quasi-shuffle products to $\diamond_L$ and $\diamond_E$ respectively. Then we have for $w,v \in \qgerH^{0}$
\begin{align*}
\sigma( \sigma(w)\hat{\shuffle}_{E} \sigma(v) ) = w \hat{\shuffle}_{L} v \,.
\end{align*}
\end{prop}
\begin{proof}
We use induction on $l(w)+l(v)$. The case $l(w)+l(v)=0$ is trivial since in this case, we have $w=v=\ew$. 
Notice that the case where $w$ or $v$ is the empty word is clear and therefore we will assume in the following that both are non-empty. In the case $l(w)+l(v)=t$,
There exists $m,n \in \Z_{\ijo 0}$, $w',v' \in \qgerH^{0}$ such that $\sigma(w)=\sigma(w')\hbar^{1-m}e_{m}$ and $\sigma(v)=\sigma(v')\hbar^{1-n}e_n$. Then, by Lemma $\ref{dualst}$
\begin{align*}
\sigma( \sigma(w)\hat{\shuffle}_{E} \sigma(v) ) 
&= \sigma( \sigma(w')\hbar^{1-m}e_{m}\hat{\shuffle}_{E}\sigma(v')\hbar^{1-n}e_n ) \\
&= \sigma\big( (\sigma(w')\hat{\shuffle}_{E}\sigma(v')\hbar^{1-n}e_n)\hbar^{1-m}e_{m} + (\sigma(w')\hbar^{1-m}e_{m}\hat{\shuffle}_{E}\sigma(v'))\hbar^{1-n}e_n \\
& \hspace{64mm} + \hbar^{2-m-n}(\sigma(w')\hat{\shuffle}_{E}\sigma(v'))(e_{m} \dia_{E} e_n) \big) \\
&= \hbar^{1-m}\sigma(e_{m})\sigma( (\sigma(w')\hat{\shuffle}_{E}\sigma(v')\hbar^{1-n}e_n)) + \hbar^{1-n}\sigma(e_{n})\sigma(\sigma(w')e_{m}\hat{\shuffle}_{E}\sigma(v'))\\  
& \hspace{71mm}+\hbar^{2-n-m} \sigma (e_{m} \dia_{E} e_n) \sigma(\sigma(w')\hat{\shuffle}_{E}\sigma(v')) \\
&= ab^{m}(w' \hat{\shuffle_{L}} v) + ab^{n}(w \hat{\shuffle_{L}} v') + \hbar^{2-m-n}\sigma(a^{m+n}(e_{0}\dia_{E}e_{0}))(w'\hat{\shuffle_{L}} v') \\
&= ab^{m}(w' \hat{\shuffle_{L}} v) + ab^{n}(w \hat{\shuffle_{L}} v') + \hbar^{2}\sigma(e_{0}\dia_{E}e_{0})b^{m+n}(w'\hat{\shuffle_{L}} v') \\
&= ab^{m}(w' \hat{\shuffle_{L}} v) + ab^{n}(w \hat{\shuffle_{L}} v') + (a \dia_{L}a )b^{m+n}(w'\hat{\shuffle_{L}} v') \\
&=  w \hat{\shuffle}_{L} v.
\end{align*}
\end{proof}

Next, we explain that the $q$-stuffle product (Definition \ref{def:stuffle}) and the $q$-shuffle product are closely related to each other. By this, we first define the following. 
\begin{df}\label{def:sigma}
We define the dual as the anti-automomorphism $\sigma$ on $\qgerH$ satisfying
 \begin{align*}
 \begin{array}{rccc}
\sigma \colon &\qgerH&\longrightarrow& \qgerH                    \\
        & a                    & \longmapsto   & \hbar b \\
        & b                    & \longmapsto   & \hbar^{-1} a \\
         & \ew                   & \longmapsto   & \ew .\\ 
\end{array}
\end{align*}
\end{df}

\begin{coro}\label{q_dual}
For $w,v \in \qgerH^{0}$, we have
\begin{align*}
w \shuffle_{q} v = \sigma (\sigma (w) \ast_{q} \sigma(v)).
\end{align*}
\end{coro}

\begin{rmk}
\begin{enumerate}[(i)]
\item Proposition $\ref{q_dual}$ is different from \cite{EMS}. There a modified Schlesinger-Zudilin $q$-analogue is used and the $\hbar$ is set to $1$. Therefore their dual $\sigma$ (denoted $\tilde{\tau}$ in \cite{EMS}) is defined by
 \begin{align*}
 \begin{array}{rccc}
\sigma \colon &\qgerH&\longrightarrow& \qgerH                    \\
        & a                    & \longmapsto   &  b \\
        & b                    & \longmapsto   &  a \\
         & \ew                   & \longmapsto   & \ew .\\ 
\end{array}
\end{align*}
\item The motivation for the definition of the generalized quasi-shuffle product came from the $q$-shuffle product. It might be interesting to know whether there are other objects whose product structure can also be obtained in terms of generalized quasi-shuffle products. For example, one might ask if there are objects whose products can be described by the generalized quasi-shuffle products coming from the Examples \ref{ex:3} or \ref{ex:4}.
\end{enumerate}
\end{rmk}

Next, we describe the multiple zeta values version of Proposition \ref{generalized_dual} through the generalized quasi-shuffle product. Furthermore, we get a different expression for the double shuffle relations.
Let $L$ be a letter of set with $L=\{ x,y \}$.
We define the commutative $\mathcal{C}$-bilinear map $\dia_{Sh}$ of fixed length by
 \begin{align*}
    \begin{array}{rccc}
\dia_{Sh} \colon &\Q L\times\Q L &\longrightarrow&    \Q         \langle L \rangle         \\
        & (x,y),(y,x)                  & \longmapsto   & -xy  \\
        & (y,y)                   & \longmapsto   & -yy  \\
        & (x,x)                   & \longmapsto   & xy. \\
           \end{array}
\end{align*}
We define $\shuffle_{Sh}$ as the generalized quasi-shuffle product on $\gerH$ by using $\dia_{Sh}$, i.e. 
\begin{align*}
\alpha w \shuffle_{Sh} \beta v = \alpha (w \shuffle_{Sh} \beta v) + \beta (\alpha w \shuffle_{Sh} v) + (\alpha \dia_{Sh} \beta)( w \shuffle_{Sh} v). 
\end{align*}
for $w,v \in \gerH$ and $\alpha, \beta \in L$. Notice that both spaces $\gerH^0$ and $\gerH^1$ are closed under $\shuffle_{Sh}$. 
We can view $\gerH$ as a subspace of $\qgerH$, by identifying $x$ with $a$ and $y$ with $b$. Setting $\hbar=1$, the action of $\tau$ is then exactly the same as $\sigma$ and we have as a consequence of Proposition \ref{generalized_dual}, that $\shuffle_{Sh}$ corresponds to the dual stuffle product: 

\begin{coro}For $w,v \in \gerH^{0}$ we have
\begin{align*}
\tau(\tau(w)*\tau(v))=w \shuffle_{Sh} v.
\end{align*}
\end{coro}

Furthermore, we consider the definition of $\shuffle_{Sh}$, we can divide $\shuffle_{Sh}$ as follows:
\begin{align*}
\alpha w \shuffle_{Sh} \beta v = \alpha w \shuffle \beta v + D(\alpha w,\beta v).
\end{align*}
For $w,v \in \gerH^{0}$ and $\alpha, \beta \in L$ the
$D(\alpha w,\beta v)$ is the collection of all  terms in $\shuffle_{Sh}$ coming from the $\dia_{Sh}$ part. Therefore, we have the following statement,
\begin{prop}
For $w,v \in \gerH^{0}$ we have
\begin{align*}
\ds(w,v)=-\tau(D(\tau(w),\tau(v))).
\end{align*}
\end{prop}
\begin{proof}
This follows from the following direct calculation
\begin{align*}
\ds(w,v)&=w\shuffle v - w*v \\
&= w\shuffle v - \tau(\tau(w)\shuffle_{Sh}\tau(v)) \\
&= w\shuffle v - \tau(\tau(w)\shuffle\tau(v) + D(\tau(w),\tau(v))) \\
&= -\tau(D(\tau(w),\tau(v))). \\
\end{align*}
Here we used that $w\shuffle v =\tau(\tau(w)\shuffle\tau(v))$, which follows from the definition of $\tau$.
\end{proof}

\begin{ex} In the case $w=v=z_2$ we get
\begin{align*}
\ds(z_2,z_2)&=-\tau (2xx(y\dia_{Sh}y)+2x(y\dia_{Sh} x)y+2(x\dia_{Sh}x)yy+(x\dia_{Sh} x)(y\dia_{Sh} y)) \\
&=-\tau(-4xxyy+xyyy)=4z_3z_1-z_4.
\end{align*}
\end{ex}

\section*{Acknowledgement}
The author would like to thank his supervisor Prof. H. Bachmann for suggesting this research theme and for providing helpful comments.

\end{document}